\documentclass[oneside,english]{amsart}
\usepackage[T1]{fontenc}
\usepackage[utf8]{inputenc}
\usepackage{amsbsy}
\usepackage{amstext}
\usepackage{amsthm}
\usepackage{amssymb}
\usepackage{esint}
\usepackage[all]{xy}

\makeatletter
\numberwithin{equation}{section}
\numberwithin{figure}{section}
\theoremstyle{plain}
\newtheorem{thm}{\protect\theoremname}[section]
  \theoremstyle{plain}
  \newtheorem{lem}[thm]{\protect\lemmaname}
  \theoremstyle{definition}
  \newtheorem{defn}[thm]{\protect\definitionname}
  \theoremstyle{plain}
  \newtheorem{cor}[thm]{\protect\corollaryname}
  \theoremstyle{plain}
  \newtheorem{prop}[thm]{\protect\propositionname}
  \theoremstyle{remark}
  \newtheorem*{claim*}{\protect\claimname}
  \theoremstyle{remark}
  \newtheorem{rem}[thm]{\protect\remarkname}


\makeatother

\usepackage{babel}
  \providecommand{\claimname}{Claim}
  \providecommand{\corollaryname}{Corollary}
  \providecommand{\definitionname}{Definition}
  \providecommand{\lemmaname}{Lemma}
  \providecommand{\propositionname}{Proposition}
  \providecommand{\remarkname}{Remark}
\providecommand{\theoremname}{Theorem}

\begin{document}

\title{The wild McKay correspondence and $p$-adic measures}

\author{Takehiko Yasuda}
\begin{abstract}
We prove a version of the wild McKay correspondence by using $p$-adic
measures. This result provides new proofs of mass formulas for extensions
of a local field by Serre, Bhargava and Kedlaya.
\end{abstract}

\address{Department of Mathematics, Graduate School of Science, Osaka University,
Toyonaka, Osaka 560-0043, Japan, tel:+81-6-6850-5326, fax:+81-6-6850-5327}

\email{takehikoyasuda@math.sci.osaka-u.ac.jp}


\maketitle
\global\long\def\AA{\mathbb{A}}
\global\long\def\PP{\mathbb{P}}
\global\long\def\NN{\mathbb{N}}
\global\long\def\GG{\mathbb{G}}
\global\long\def\ZZ{\mathbb{Z}}
\global\long\def\QQ{\mathbb{Q}}
\global\long\def\CC{\mathbb{C}}
\global\long\def\FF{\mathbb{F}}
\global\long\def\LL{\mathbb{L}}
\global\long\def\RR{\mathbb{R}}
\global\long\def\MM{\mathbb{M}}
\global\long\def\SS{\mathbb{S}}

\global\long\def\bx{\mathbf{x}}
\global\long\def\bf{\mathbf{f}}
\global\long\def\ba{\mathbf{a}}
\global\long\def\bs{\mathbf{s}}
\global\long\def\bt{\mathbf{t}}
\global\long\def\bw{\mathbf{w}}
\global\long\def\bb{\mathbf{b}}
\global\long\def\bv{\mathbf{v}}
\global\long\def\bp{\mathbf{p}}
\global\long\def\bm{\mathbf{m}}
\global\long\def\bj{\mathbf{j}}
\global\long\def\bM{\mathbf{M}}
\global\long\def\bz{\boldsymbol{z}}

\global\long\def\cN{\mathcal{N}}
\global\long\def\cW{\mathcal{W}}
\global\long\def\cY{\mathcal{Y}}
\global\long\def\cM{\mathcal{M}}
\global\long\def\cF{\mathcal{F}}
\global\long\def\cX{\mathcal{X}}
\global\long\def\cE{\mathcal{E}}
\global\long\def\cJ{\mathcal{J}}
\global\long\def\cO{\mathcal{O}}
\global\long\def\cD{\mathcal{D}}
\global\long\def\cZ{\mathcal{Z}}
\global\long\def\cR{\mathcal{R}}
\global\long\def\cC{\mathcal{C}}
\global\long\def\cU{\mathcal{U}}
\global\long\def\cI{\mathcal{I}}

\global\long\def\fs{\mathfrak{s}}
\global\long\def\fp{\mathfrak{p}}
\global\long\def\fm{\mathfrak{m}}
\global\long\def\fX{\mathfrak{X}}
\global\long\def\fV{\mathfrak{V}}
\global\long\def\fx{\mathfrak{x}}
\global\long\def\fv{\mathfrak{v}}
\global\long\def\fY{\mathfrak{Y}}

\global\long\def\rv{\mathbf{\mathrm{v}}}
\global\long\def\rx{\mathrm{x}}
\global\long\def\rw{\mathrm{w}}
\global\long\def\ry{\mathrm{y}}
\global\long\def\rz{\mathrm{z}}
\global\long\def\bv{\mathbf{v}}
\global\long\def\bx{\boldsymbol{x}}
\global\long\def\by{\boldsymbol{y}}
\global\long\def\sv{\mathsf{v}}
\global\long\def\sx{\mathsf{x}}
\global\long\def\sw{\mathsf{w}}
\global\long\def\bomega{\boldsymbol{\omega}}
\global\long\def\bphi{\boldsymbol{\phi}}
\global\long\def\bpsi{\boldsymbol{\psi}}

\global\long\def\Spec{\mathrm{Spec}\,}
\global\long\def\Hom{\mathrm{Hom}}
\global\long\def\Spf{\mathrm{\mathrm{Spf}\,}}

\global\long\def\Var{\mathrm{Var}}
\global\long\def\Gal{\mathrm{Gal}}
\global\long\def\Jac{\mathrm{Jac}}
\global\long\def\Ker{\mathrm{Ker}}
\global\long\def\Im{\mathrm{Im}}
\global\long\def\Aut{\mathrm{Aut}}
\global\long\def\st{\mathrm{st}}
\global\long\def\diag{\mathrm{diag}}
\global\long\def\characteristic{\mathrm{char}}
\global\long\def\tors{\mathrm{tors}}
\global\long\def\sing{\mathrm{sing}}
\global\long\def\red{\mathrm{red}}
\global\long\def\Ind{\mathrm{Ind}}
\global\long\def\nr{\mathrm{nr}}
\global\long\def\ord{\mathrm{ord}}
\global\long\def\pt{\mathrm{pt}}
\global\long\def\op{\mathrm{op}}
 \global\long\def\univ{\mathrm{univ}}
\global\long\def\length{\mathrm{length}}
\global\long\def\sm{\mathrm{sm}}
\global\long\def\top{\mathrm{top}}
\global\long\def\rank{\mathrm{rank}}
\global\long\def\Mot{\mathrm{Mot}}
\global\long\def\age{\mathrm{age}\,}
\global\long\def\et{\mathrm{et}}
\global\long\def\hom{\mathrm{hom}}
\global\long\def\tor{\mathrm{tor}}
\global\long\def\reg{\mathrm{reg}}
\global\long\def\an{\mathrm{an}}
\global\long\def\nor{\mathrm{nor}}

\global\long\def\Conj#1{\mathrm{Conj}(#1)}
\global\long\def\Mass#1{\mathrm{Mass}(#1)}
\global\long\def\Inn#1{\mathrm{Inn}(#1)}
\global\long\def\bConj#1{\mathbf{Conj}(#1)}
\global\long\def\Hilb{\mathrm{Hilb}}
\global\long\def\sep{\mathrm{sep}}
\global\long\def\GL#1#2{\mathrm{GL}_{#1}(#2)}
\global\long\def\codim{\mathrm{codim}}
\global\long\def\stw{\mathrm{stw}}

\global\long\def\GEt{G\text{-}\mathrm{\acute{E}t}(K)}
\global\long\def\GCov{G\text{-}\mathrm{Cov}}
\global\long\def\preuntwisting{\left\langle F\right\rangle }
\global\long\def\Tr{\mathrm{Tr}}
\global\long\def\Nr{\mathrm{Nr}}
\global\long\def\Eta{\mathrm{Eta}}
\global\long\def\Fie{\mathrm{Fie}}
\global\long\def\Bl{\mathrm{Bl}}
\global\long\def\SnEt{S_{n}\text{-}\mathrm{\acute{E}t}(K)}
\global\long\def\nEt{n\text{-}\mathrm{\acute{E}t}(K)}

\tableofcontents{}

\section{Introduction\label{sec:Introduction}}

The aim of this paper is to prove a version of the wild McKay correspondence,
the McKay correspondence in positive or mixed characteristic where
a given finite group may have order dividing the characteristic of
the base field or the residue field. Our main tool is the $p$-adic
measure.

By the McKay correspondence, we mean an equality between a certain
invariant of a $G$-variety $V$ with $G$ a finite group and a similar
invariant of the quotient variety $V/G$ or a desingularization of
it. There are different versions for different invariants. Our concern
is the one using motivic invariants or their realizations. In characteristic
zero, such a version was studied by Batyrev \cite{MR1677693} and
Denef-Loeser \cite{MR1905024}. Recently, after examining a special
case in \cite{MR3230848}, the author started to try to generalize
it to positive or mixed characteristic, and formulated a conjecture
in \cite{Yasuda:2013fk} for linear actions on affine spaces over
a complete discrete valuation ring with algebraically closed residue
field. Later, variants and generalizations were formulated in \cite{Wood-Yasuda-I,Yasuda:2014fk2}.
In \cite{Wood-Yasuda-I}, the situation was considered where the residue
field is only perfect. Moreover, when the residue field is finite,
the point-counting realization was discussed. In \cite{Yasuda:2014fk2},
non-linear actions on affine normal varieties were treated. In the
present paper, we consider non-linear actions on normal quasi-projective
varieties over a complete discrete valuation ring with finite residue
field and prove a version of the wild McKay correspondence at the
level of point-counting realization, with a little dissatisfaction
at the formulation in the non-affine case.

Let $\cO_{K}$ be a complete discrete valuation ring, $K$ its fraction
field and $k$ its residue field, which is supposed to be finite.
For the pair $(X,D)$ of an $\cO_{K}$-variety $X$ and a $\QQ$-divisor
$D$ on $X$ such that $K_{X}+D$ is $\QQ$-Cartier with $K_{X}$
the canonical divisor of $X$ over $\cO_{K}$, we define the \emph{stringy
point count }of $(X,D)$, 
\[
\sharp_{\st}(X,D)\in\RR_{\ge0}\cup\{\infty\},
\]
as the volume of $X(\cO_{K})$ with respect to a certain $p$-adic
measure. When $D=0$, identifying the pair $(X,0)$ with the variety
$X$ itself, we write $\sharp_{\st}(X,0)=\sharp_{\st}X$. Roughly,
the stringy point count is the point-count realization of the motivic
counterpart of the stringy $E$-function introduced by Batyrev \cite{MR1672108,MR1677693}.
Its principal properties are as follows.
\begin{itemize}
\item When $X$ is $\cO_{K}$-smooth, we have $\sharp_{\st}X=\sharp X(k)$. 
\item There exists a decomposition into contributions of $k$-points,
\[
\sharp_{\st}(X,D)=\sum_{x\in X(k)}\sharp_{\st}(X,D)_{x}.
\]

\item If $f:Y\to X$ is a proper birational morphism of normal $\cO_{K}$-varieties
which induces a crepant map $(Y,E)\to(X,D)$ of pairs, then we have
\[
\sharp_{\st}(X,D)=\sharp_{\st}(Y,E).
\]

\end{itemize}
We generalize the invariant to pairs having a finite group action.
Let $(V,E)$ be a pair as above and suppose that a finite group $G$
faithfully acts on $V$ and the divisor $E$ is stable under the action.
Let $M$ be a $G$-étale $K$-algebra, that is, $\Spec M\to\Spec K$
is an étale $G$-torsor and let $\cO_{M}$ be its integer ring. We
define the $M$-\emph{stringy point count }of $(V,E)$, denoted by
$\sharp_{\st}^{M}(V,E)$, as a certain $p$-adic volume of the set
of $G$-equivariant $\cO_{K}$-morphisms $\Spec\cO_{M}\to V$, and
define the $G$\emph{-stringy point count }by\emph{
\[
\sharp_{\st}^{G}(V,E)=\sum_{M}\sharp_{\st}^{M}(V,E),
\]
}where $M$ runs over the isomorphism classes of $G$-étale $K$-algebras.
Thus $\sharp_{\st}^{G}(V,E)$ is the weighted count of $G$-étale
$K$-algebras $M$ with weights $\sharp_{\st}^{M}(V,E)$. Basic properties
of the $G$-stringy point count are as follows.
\begin{itemize}
\item When $G=1$, we have $\sharp_{\st}^{G}(V,E)=\sharp_{\st}(V,E)$. 
\item When $V=\AA_{\cO_{K}}^{n}$, $E=0$ and the $G$-action is linear,
then 
\[
\sharp_{\st}^{M}V=\frac{q^{n-\bv_{V}(M)}}{\sharp\Aut^{G}(M/K)},
\]
where $\bv_{V}$ is a function associated to the $G$-action on $V$
and $\Aut^{G}(M/K)$ is the group of $G$-equivariant $K$-automorphisms
of $M$. 
\end{itemize}
The invariant $\sharp_{\st}^{G}(V,E)$ is roughly the point-counting
realization of a motivic invariant studied in \cite{Yasuda:2014fk2},
which is a simultaneous refinement and generalization of the orbifold
$E$-function and the stringy $E$-function considered by Batyrev
\cite{MR1677693}. 

If we put $X:=V/G$, there exists a unique $\QQ$-divisor $D$ on
$X$ such that the natural morphism $(V,E)\to(X,D)$ of pairs is crepant.
Our main result is as follows.
\begin{thm}
We have
\[
\sharp_{\st}(X,D)=\sharp_{\st}^{G}(V,E).
\]

\end{thm}
This is the point-counting version of a conjecture in \cite{Yasuda:2014fk2}.
The proof basically follows the strategy presented in \cite{Yasuda:2013fk,Yasuda:2014fk2},
which generalize arguments in characteristic zero by Denef-Loeser
\cite{MR1905024}, except that we use $p$-adic measures instead of
motivic integration. This switch, from motives to numbers of points,
and from motivic integration to $p$-adic measures, enables us to
avoid the use of conjectural moduli spaces which the author relied
on in \cite{Yasuda:2013fk,Yasuda:2014fk2}. It also makes a large
part of the arguments much simpler. Although the author believes that
we would have the desired moduli spaces and prove more general and
stronger results by means of motivic integration in near future, it
would be nice to have an elementary and short proof of a result a
little weaker but still strong enough for many applications. In the
text, we prove a slightly more general result than the theorem above:
we prove $\sharp_{\st}(X,D)_{\overline{C}}=\sharp_{\st}^{G}(V,E)_{C}$
for a $G$-stable constructible subset $C$ of $V\otimes_{\cO_{K}}k$
and its image $\overline{C}$ in $X\otimes_{\cO_{K}}k$.

It is suggestive to write the equality of the theorem as 
\[
\sum_{x\in X(k)}\sharp_{\st}(X,D)_{x}=\sum_{M}\sharp_{\st}^{M}(V,E),
\]
an equality between a weighted count of $k$-points of $X$ and one
of $G$-étale $K$-algebras. A particularly interesting situation
of the theorem is as follows. We suppose that $V=\AA_{\cO_{K}}^{n}$,
that the $G$-action is linear and has no pseudo-reflection, and that
there exists a crepant proper birational morphism $Y\to X$ with $Y$
regular. If we denote the $\cO_{K}$-smooth locus of $Y$ by $Y_{\sm}$,
then the theorem reduces to the form 
\begin{equation}
\sharp Y_{\sm}(k)=\sum_{M}\frac{q^{n-\bv_{V}(M)}}{\sharp\Aut^{G}(M/K)}.\label{mass}
\end{equation}
When it is possible to count $k$-points of $Y_{\sm}$ explicitly,
we readily obtain a \emph{mass formula} for $G$-étale $K$-algebras
with respect to weights $\frac{q^{n-\bv_{V}(M)}}{\sharp\Aut^{G}(M/K)}$. 

Serre \cite{MR500361} proved a beautiful mass formula for totally
ramified field extensions $L/K$ of fixed degree with respect to weights
determined by discriminants. He gave two different proofs. Krasner
\cite{zbMATH03663278} gave an alternative proof by using a formula
which had been obtained by himself. As far as the author knows, these
have been all known proofs of Serre's mass formula. Bhargava \cite{MR2354798}
proved a similar mass formula for all étale $K$-algebras of fixed
degree, using Serre's formula. Kedlaya \cite{MR2354797} interpreted
Bhargava's formula as a mass formula for local Galois representations
(that is, continuous homomorphisms $\Gal(K^{\sep}/K)\to S_{n}\subset\GL n{\CC}$)
with respect to the Artin conductor. He then studied the case where
$S_{n}$ is replaced with other groups. 

Wood and Yasuda \cite{Wood-Yasuda-I} showed a close relation between
the function $\bv_{V}$ and the Artin conductor. Using this and a
desingularization by the Hilbert scheme of points, we can deduce Bhargava's
formula as a special case of formula (\ref{mass}). A similar relation
between Bhargava's formula and the Hilbert scheme of points was discussed
in \cite{Wood-Yasuda-I}. Then, using for instance the exponential
formula relating Serre's and Bhargava's formulas \cite[page 8]{MR2354797},
we can give a new proof of Serre's formula. In a similar way, we can
also prove Kedlaya's mass formula for the group of signed permutation
matrices in $\GL n{\CC}$ \cite{MR2354797}, unless $K$ has residual
characteristic two. Detailed computation of the last example will
be given in \cite{Wood-Yasuda-II}. These new proofs of mass formulas
are not as easy as the original ones. However they are interesting
because they fit into the general framework of the wild McKay correspondence
and reduce the problem to explicit computation of desingularization,
which seems unrelated at first glance.

The paper is organized as follows. In section \ref{sec:Convention-and-notation}
we set our convention and notation. In section \ref{sec:p-adic-measrues}
we recall $K$-analytic manifolds with $K$ a local field and $p$-adic
measures on them associated to differential forms. In section \ref{sec:log pairs}
we define stringy point counts of log pairs. In section \ref{sec:Group-actions}
we show a certain one-to-one correspondence of points associated to
a Galois cover of varieties. In section \ref{sec:Untwisting} we discuss
the untwisting technique, which is the technical core of the proof
of our main result. In section \ref{sec:Main-results} we prove the
main result. In section \ref{sec:mass formulas} we discuss the case
where a finite group linearly acts on an affine space and its application
to mass formulas.

\subsection*{Acknowledgments}

The author would like to thank Tomoyoshi Ibukiyama and Seidai Yasuda
for letting me know Serre's mass formula, and Melanie Matchett Wood
for helpful discussions during their joint works. He is also indebted
to two referees for reading carefully the paper and suggestions for
improvement. In particular, the present proof of Lemma \ref{lem:null-set-1},
which is simpler than the previous one, follows a suggestion by one
of them.

\section{Convention and notation\label{sec:Convention-and-notation}}

\subsection{}

Throughout the paper, we fix a non-archimedean local field $K$, that
is, a finite extension of either $\QQ_{p}$ or $\FF_{p}((t))$. We
denote its integer ring by $\cO_{K}$, its residue field by $k$,
the cardinality of $k$ by $q$, and the maximal ideal of $\cO_{K}$
by $\fm_{K}$.

\subsection{}

We usually denote by $M$ an $G$-étale $K$-algebra (section \ref{sub:G-alg})
and by $\Spec L$ a connected component of $\Spec M$ so that $L$
is a finite separable field extension of $K$. We denote by $\cO_{M}$
and $\cO_{L}$ the rings of integers of $M$ and $L$ respectively.
We denote by $H$ the stabilizer subgroup of $G$ of this component.

\subsection{}

If $R$ is either $K$, $\cO_{K}$ or $k$, an $R$\emph{-variety
}means a reduced quasi-projective $R$-scheme $X$ such that
\begin{itemize}
\item $X$ is flat and of finite type over $R$,
\item $X$ is equi-dimensional over $R$; all irreducible components have
the same relative dimension over $R$, and
\item the structure morphism $X\to\Spec R$ is smooth on an open dense subscheme
of $X$. 
\end{itemize}
The \emph{dimension }of an $R$-variety always means its relative
dimension over $R$. We usually denote the dimension of a variety
by $d$ (or $n$ when the variety is an affine space). For an $\cO_{K}$-variety
$X$, we define $X_{K}:=X\otimes_{\cO_{K}}K$ and $X_{k}:=X\otimes_{\cO_{K}}k$.
Note that $X_{k}$ is not generally reduced or a $k$-variety.

\subsection{}

For an $\cO_{K}$-variety $X$, we denote by $X(\cO_{K})^{\circ}$
the set of the $\cO_{K}$-points $\Spec\cO_{K}\to X$ that send the
generic point of $\Spec\cO_{K}$ into the locus where $X$ is $\cO_{K}$-smooth.

\subsection{}

Groups act on schemes from left and on rings, fields and modules from
right, unless otherwise noted. Thus, for an affine scheme $\Spec R$,
if a group $G$ acts on $\Spec R$ and if $g:\Spec R\to\Spec R$ is
the automorphism induced by $g\in G$, then we have the corresponding
ring automorphism $g^{*}:R\to R$ and the same group $G$ naturally
acts on $R$ by $r\cdot g:=g^{*}(r)$, $r\in R$. Conversely, a $G$-action
on $R$ gives a natural $G$-action on $\Spec R$ in a similar way.

\section{$p$-adic measures on $K$-analytic manifolds\label{sec:p-adic-measrues}}

In this section, we review basic materials on the Haar measure on
$K^{d}$ for a local field $K$ and a measure on a $K$-analytic manifold
induced by a differential form.

\subsection{\label{sub:K}}

Let $K$ be a non-archimedean local field, that is, a finite field
extension of $\QQ_{p}$ or $\FF_{p}((t))$. We denote its integer
ring and residue field by $\cO_{K}$ and $k$ respectively. We always
denote the cardinality of $k$ by $q$. Let $|\cdot|$ be the normalized
absolute value on $K$ so that $|\varpi|=q^{-1}$ for a uniformizer
$\varpi\in\cO_{K}$. For an integer $d\ge0$, we define $\mu_{K^{d}}$
to be the Haar measure of $K^{d}$ normalized so that $\mu_{K^{d}}(\cO_{K}^{d})=1$.
A function $f:U\to K$ on an open subset $U\subset K^{d}$ is called
$K$-\emph{analytic }if in a neighborhood of every point of $U$,
$f$ is expressed as a convergent Taylor series. 
\begin{lem}
\label{lem:null-set-1}Let $f:U\to K$ be a $K$-analytic function
defined on an open compact subset $U\subset K^{d}$. Suppose that
$f$ is nowhere locally constant. Then $\mu_{K^{d}}(f^{-1}(0))=0$.\end{lem}
\begin{proof}
This result should be well-known to specialists, though the author
could not find a reference. For the sake of completeness, we give
a proof here, which follows a suggestion of a referee. (A similar
result is found in Igusa's book \cite[Lemma 8.3.1]{MR1743467}. However
he proves it only for the characteristic zero case.) 

The proof is by induction on $d$. If $d=0$, there is nothing to
prove. If $d=1$, then looking at the Taylor expansion, we see that
$f^{-1}(0)$ is a discrete subset of $U$ and has measure zero. For
$d>1$, it suffices to show that $f^{-1}(0)$ contains the origin
$o\in K^{d}$ and there exists an open neighborhood $V$ of $o$ such
that $f^{-1}(0)\cap V$ has measure zero. There exists a line $L\cong K$
passing through $o$ such that $f|_{L\cap U}$ is not locally constant
around $o$. Indeed, if there does not exist such a line and if $f$
is expressed as a power series on a neighborhood $V$ of $o$, then
$f$ is constant on $V$, which contradicts the assumption. By a suitable
linear transform, we may assume that $L$ is given by $x_{1}=\cdots=x_{d-1}=0$.
From the Weierstrass preparation theorem (for instance, see \cite[Theorem 2.3.1]{MR1743467}),
$f$ is of the form
\[
g(x_{1},\dots,x_{d})\left(x_{d}^{m}+f_{1}(x_{1},\dots,x_{d-1})x_{d}^{m-1}+\cdots+f_{m}(x_{1},\dots,x_{d-1})\right)
\]
on a neighborhood $V$ of $o$, where $g,f_{1},\dots,f_{m}$ are convergent
power series such that $g$ is nowhere vanishing on $V$. Define $h:V\to K$
by 
\[
h:=f/g=x_{d}^{m}+f_{1}(x_{1},\dots,x_{d-1})x_{d}^{m-1}+\cdots+f_{m}(x_{1},\dots,x_{d-1}),
\]
which has the same zero locus as $f|_{V}$. Let $\pi:V\to K^{d-1}$
be the projection to the first $d-1$ coordinates. For every $y\in\pi(V)$,
$h|_{\pi^{-1}(y)}$ is a function given by a nonzero polynomial, in
particular, it is nowhere locally constant. From the case $d=1$,
its zero locus has measure zero with respect to $\mu_{K}$. Now the
Fubini-Tonelli theorem (for instance, see \cite[Theorems A and B, page 147]{MR0033869})
shows that $f^{-1}(0)\cap V=h^{-1}(0)$ has measure zero with respect
to $\mu_{K^{d}}$, which completes the proof. 
\end{proof}

\subsection{\label{sub:K-an def}}

$K$-\emph{analytic manifolds} are defined in a similar way as ordinary
manifolds are defined. For details, we refer the reader to \cite[section 2.4]{MR1743467}.
We can similarly define $K$-\emph{analytic differential forms }as
well. Let $X$ be a $K$-analytic manifold of dimension $d$. For
a $K$-analytic $d$-form $\omega$ on $X$ and an open compact subset
$U$ of $X$, one can define the integral 
\[
\int_{U}|\omega|\in\RR_{\ge0}.
\]
When $\omega$ is written as $f(x)dx_{1}\wedge\cdots\wedge dx_{d}$
for a $K$-analytic function $f(x)$ and local coordinates $x_{1},\dots,x_{d}$
on $U$, then 
\[
\int_{U}|\omega|=\int_{U'}|f(x)|\,d\mu_{K^{d}},
\]
where $U'$ is the open subset of $K^{n}$ corresponding to $U$.
Thus a $K$-analytic $d$-form defines a measure $\mu_{\omega}$ on
$X$: for a compact open subset $U\subset X$, 
\[
\mu_{\omega}(U):=\int_{U}|\omega|\in\RR_{\ge0}
\]
and for an arbitrary open subset $U\subset X$, 
\[
\mu_{\omega}(U):=\sup\{\mu_{\omega}(U')\mid U'\subset U\text{: open and compact}\}\in\RR_{\ge0}\cup\{\infty\}.
\]

For our purpose, we need to generalize this slightly more as did in
\cite{MR2098399,MR1678489}. Let $\omega$ be an $r$-fold $d$-form,
that is, a section of $(\Omega_{X}^{d})^{\otimes r}$. Here $\Omega_{X}^{d}$
is the sheaf of $K$-analytic $d$-forms and the tensor product is
taken over the sheaf of $K$-analytic functions. Locally $\omega$
is written as $f(x)(dx_{1}\wedge\cdots\wedge dx_{d})^{\otimes r}$
with $f(x)$ a $K$-analytic function, say on an open compact subset
$U$. We then define
\[
\int_{U}|\omega|^{1/r}:=\int_{U'}|f(x)|^{1/r}\,d\mu_{K^{d}}
\]
and extend the definition to an arbitrary $r$-fold $d$-form on $X$
in the obvious way. We define the measure $\mu_{\omega}$ by 
\[
\mu_{\omega}(U):=\int_{U}|\omega|^{1/r}
\]
for an open compact $U$, similarly for an arbitrary open subset.

\section{Log pairs\label{sec:log pairs}}

Using measures on $K$-analytic manifolds considered in the preceding
section, we introduce, in this section, the notion of \emph{stringy
point count }for log pairs and study its basic properties.

\subsection{\label{sub:ideal sheaf}}

Let $X$ be a $d$-dimensional $\cO_{K}$-variety. We write $X_{k}:=X\otimes_{\cO_{K}}k$
and $X_{K}:=X\otimes_{\cO_{K}}K$. Let $X_{K,\sm}$ be the $K$-smooth
locus of $X_{K}$ and let 
\[
X(\cO_{K})^{\circ}:=X(\cO_{K})\cap X_{K,\sm}(K),
\]
thinking of $X(\cO_{K})$ as a subset of $X(K)=X_{K}(K)$. This set
$X(\cO_{K})^{\circ}$ has a natural structure of a $K$-analytic manifold. 

Let $\cI$ be an invertible $\cO_{X}$-submodule of $(\Omega_{X/\cO_{K}}^{d})^{\otimes r}\otimes K(X)$,
where $K(X)$ is the sheaf of total quotient rings of $X$. We define
a measure $\mu_{\cI}$ on $X(\cO_{K})^{\circ}$ as follows. Let $X=\bigcup U_{i}$
be a Zariski open cover so that $X(\cO_{K})^{\circ}=\bigcup U_{i}(\cO_{K})^{\circ}$
and $\cI|_{U_{i}}$ is a free $\cO_{U_{i}}$-module. Let $\omega_{i}\in\cI|_{U_{i}}$
be a generator. It defines an $r$-fold $d$-form $\omega_{i}^{\an}$
on the $K$-analytic manifold $U_{i}(\cO_{K})^{\circ}$ in the obvious
way and the measure $\mu_{\omega_{i}^{\an}}$. If $\omega'_{i}$ is
another generator of $\cI|_{U_{i}}$, then there exists a nowhere
vanishing regular function $f$ on $U_{i}$ such that $\omega_{i}=f\omega_{i}'$.
If $f^{\an}$ is the corresponding $K$-analytic function on $U_{i}(\cO_{K})^{\circ}$,
then $|f^{\an}|\equiv1$. Therefore the measure $\mu_{\omega_{i}^{\an}}$
does not depend on the choice of generator. Now it is clear that the
measures $\mu_{\omega_{i}^{\an}}$ for different $i$ glue together
and define a measure on the entire space $X(\cO_{K})^{\circ}$: we
denote it by $\mu_{\cI}$. We further extend the measure $\mu_{\cI}$
to $X(\cO_{K})$, declaring that all subsets of $X(\cO_{K})\setminus X(\cO_{K})^{\circ}$
have measure zero.

The following lemma is a slight generalization of \cite[Theorem 2.2.5]{MR670072}.
\begin{lem}
\label{lem:Weil estimate}If $X$ is $\cO_{K}$-smooth and $\cI=(\Omega_{X/\cO_{K}}^{d})^{\otimes r}$,
then
\[
\mu_{\cI}(X(\cO_{K}))=\frac{\sharp X(k)}{q^{d}}.
\]
\end{lem}
\begin{proof}
For $x\in X(k)$, let $X(\cO_{K})_{x}$ be the set of $\cO_{K}$-point
which induce $x$ by composition with $\Spec k\to\Spec\cO_{K}$. If
$x_{1},\dots,x_{d}$ are local coordinates around $x$, then they
give a bijection from $X(\cO_{K})_{x}$ onto $\fm_{K}^{d}\subset K^{d}$.
On the other hand, $\cI$ has a local generator $(dx_{1}\wedge\cdots\wedge dx_{d})^{\otimes r}$.
Therefore 
\[
\mu_{\cI}(X(\cO_{K})_{x})=\int_{\fm_{K}^{d}}1\,d\mu_{K^{d}}=\mu_{K^{d}}(\fm_{K}^{d})=q^{-d}
\]
and the lemma follows.\end{proof}
\begin{lem}
\label{lem:power ideal}For an integer $s>0$, regarding $\cI^{\otimes s}$
as an $\cO_{K}$-submodule of $(\Omega_{X/\cO_{K}}^{d})^{\otimes sr}\otimes K(X)$,
we have $\mu_{\cI}=\mu_{\cI^{\otimes s}}.$\end{lem}
\begin{proof}
If $\cI$ has a local generator $f(x)(dx_{1}\wedge\cdots\wedge dx_{d})^{\otimes r}$,
then $f(x)^{s}(dx_{1}\wedge\cdots\wedge dx_{d})^{\otimes sr}$ is
a local generator of $\cI^{\otimes s}$, and we have
\[
\int|f(x)|^{1/r}\,d\mu_{K^{d}}=\int|f(x)^{s}|^{1/(rs)}\,d\mu_{K^{d}},
\]
where the integrals are taken over a suitable open compact subset
of $K^{d}$. The lemma easily follows.\end{proof}
\begin{lem}
\label{lem:measure zero 2}For a subscheme $Y\subset X$ of positive
codimension, 
\[
\mu_{\cI}(Y(\cO_{K}))=0.
\]
\end{lem}
\begin{proof}
This is a direct consequence of Lemma \ref{lem:null-set-1}.
\end{proof}

\subsection{\label{sub:log pair}}

For a normal $\cO_{K}$-variety $X$, the canonical sheaf $\omega_{X}=\omega_{X/\cO_{K}}$
is defined as in \cite[pages 7--8]{MR3057950}. It is a reflexive
sheaf, in particular, locally free in codimension one, and coincides
with $\Omega_{X/\cO_{K}}^{d}:=\bigwedge^{d}\Omega_{X/\cO_{K}}$ on
the $\cO_{K}$-smooth locus. We denote the corresponding divisor by
$K_{X}$, which is determined up to linear equivalence. 
\begin{defn}
A \emph{log pair }means the pair $(X,D)$ of a normal $\cO_{K}$-variety
$X$ and a $\QQ$-divisor $D$ (a Weil divisor with rational coefficients)
on $X$ such that $K_{X}+D$ is $\QQ$-Cartier. We sometimes call
a log pair $(X,D)$ a \emph{log structure }on $X$.
\end{defn}
We identify a normal $\QQ$-Gorenstein ($K_{X}$ is $\QQ$-Cartier)
$\cO_{K}$-variety $X$ with the log pair $(X,0)$. 

Let $(X,D)$ be a log pair and let $r\in\NN$ be such that $r(K_{X}+D)$
is Cartier. Then the invertible sheaf $\cO_{X}(r(K_{X}+D))$ is naturally
a subsheaf of $(\Omega_{X/\cO_{K}}^{d})^{\otimes r}\otimes K(X)$. 
\begin{defn}
A \emph{morphism }of log pairs, $f:(Y,E)\to(X,D)$, is a morphism
of the underlying varieties, $f:Y\to X$. We say that a morphism $f:(Y,E)\to(X,D)$
is \emph{crepant }if for $r\in\NN$ such that $r(K_{X}+D)$ and $r(K_{Y}+E)$
are both Cartier, the canonical morphism $f^{*}(\Omega_{X/\cO_{K}}^{d})^{\otimes r}\to(\Omega_{Y/\cO_{K}}^{d})^{\otimes r}$
induces an isomorphism $f^{*}\cO_{X}(r(K_{X}+D))\to\cO_{Y}(r(K_{Y}+E))$. \end{defn}
\begin{lem}
Let $(X,D)$ be a log pair and let $f:Y\to X$ be a generically étale
morphism of normal varieties. There exists a unique $\QQ$-divisor
$E$ on $Y$ such that the morphism $f:(Y,E)\to(X,D)$ is crepant.\end{lem}
\begin{proof}
The pull-back $f^{*}\cO_{X}(r(K_{X}+D))$ is an invertible subsheaf
of $(\Omega_{Y/\cO_{K}}^{d})^{\otimes r}\otimes K(Y)$. Hence there
exists a unique Weil divisor $E'$ such that $\omega_{Y/\cO_{K}}^{\otimes r}(E')$
coincides with $f^{*}\cO_{X}(r(K_{X}+D))$ in codimension one as subsheaves
of $(\Omega_{Y/\cO_{K}}^{d})^{\otimes r}\otimes K(Y)$. Now $\frac{1}{r}E'$
is the desired $\QQ$-divisor.
\end{proof}

\subsection{\label{sub:log measure}}
\begin{defn}
Let $(X,D)$ be a log pair and $r\in\NN$ such that $r(K_{X}+D)$
is Cartier. We define the measure $\mu_{X,D}$ on $X(\cO_{K})$ to
be $\mu_{\cO_{X}(r(K_{X}+D))}$. For a normal $\QQ$-Gorenstein $\cO_{K}$-variety
$X$, we write $\mu_{X,0}$ simply as $\mu_{X}$. 
\end{defn}
From Lemma \ref{lem:power ideal}, the definition does not depend
on the choice of $r$. If $X$ is $\cO_{K}$-smooth, from Lemma \ref{lem:Weil estimate},
we have
\[
\mu_{X}(X(\cO_{K}))=\frac{\sharp X(k)}{q^{d}}.
\]

We use the following theorem many times in later sections:
\begin{thm}
\label{thm:change-vars}Let $f:(Y,E)\to(X,D)$ be a crepant morphism
of log pairs. Suppose that a finite group $G$ acts faithfully on
$Y$ and trivially on $X$ so that the morphism $f:Y\to X$ of the
underlying varieties is $G$-equivariant. Suppose that the induced
morphism $Y/G\to X$ is birational. Then, for a $G$-stable open subset
$A\subset Y(\cO_{K})$, we have 
\[
\frac{1}{\sharp G}\mu_{Y,E}(A)=\mu_{X,D}(f(A))\in\RR_{\ge0}\cup\{\infty\}.
\]
\end{thm}
\begin{proof}
From Lemma \ref{lem:measure zero 2}, removing measure zero subsets
from $A$ and $f(A)$, we may suppose that the map $A\to f(A)$ is
a $G$-equivariant étale morphism of $K$-analytic manifolds with
$A/G\cong f(A)$. Each point $x\in f(A)$ has a $K$-analytic open
neighborhood $U$ such that $f^{-1}(U)\subset A$ is isomorphic to
the disjoint union of $\sharp G$ copies of $U$. Let $r\in\NN$ be
such that $r(K_{X}+D)$ is Cartier and $\cI:=\cO_{X}(r(K_{X}+D))$.
For a Zariski open $V\subset X$ such that $U\subset V(\cO_{K})$,
if $\omega\in\Gamma(V,\cI)$ is a generator of $\cI|_{V}$, we clearly
have
\[
\frac{1}{\sharp G}\int_{f^{-1}(U)}|f^{-1}\omega|^{1/r}=\int_{U}|\omega|^{1/r}.
\]
It is now easy to deduce the theorem.
\end{proof}

\subsection{\label{sub:stringy count}}
\begin{defn}
Let $(X,D)$ be a log pair, $d$ the dimension of the $\cO_{K}$-variety
$X$ and $C\subset X_{k}$ a constructible subset. Let $X(\cO_{K})_{C}$
be the set of $\cO_{K}$-points sending the closed point of $\Spec\cO_{K}$
into $C$. We define the \emph{stringy point count }of a log pair
$(X,D)$\emph{ }along $C$ by 

\[
\sharp_{\st}(X,D)_{C}:=q^{d}\cdot\mu_{X,D}(X(\cO_{K})_{C}).
\]
When $C=X_{k}$, we omit the subscript $C$ and write $\sharp_{\st}(X,D)$.
When $C$ consists of a single $k$-point $x$, we write the subscript
simply as $x$.
\end{defn}
We obviously have
\[
\sharp_{\st}(X,D)=\sum_{x\in X(k)}\sharp_{\st}(X,D)_{x},
\]
showing that $\sharp_{\st}(X,D)$ is the count of $k$-points with
weights 
\[
\sharp_{\st}(X,D)_{x}.
\]
If $X$ is a smooth variety, identified with the log pair $(X,0)$,
then 
\[
\sharp_{\st}X=\sharp X(k).
\]

\begin{cor}
\label{cor:modification}Let $f:(Y,E)\to(X,D)$ be a crepant proper
birational morphism of log pairs and $C\subset X_{k}$ a constructible
subset. We have
\[
\sharp_{\st}(Y,E)_{f^{-1}(C)}=\sharp_{\st}(X,D)_{C}.
\]
\end{cor}
\begin{proof}
We apply Theorem \ref{thm:change-vars} to the case where $G=1$ and
$A=Y(\cO_{K})_{f^{-1}(C)}$. From the valuative criterion for properness,
$f(A)$ coincides with $X(\cO_{K})_{C}$ modulo measure zero subsets,
and the corollary follows.
\end{proof}

\subsection{\label{sub:explicit formula}}

We now give an explicit formula for stringy point counts under a certain
assumption, an analogue of Denef's formula \cite[Theorem 3.1]{MR919001}
and the definition of the stringy $E$-function \cite{MR1672108}.
Although we do not use it in the rest of the paper, the formula is
useful for applications. Firstly we suppose that $X$ is regular.
We also suppose that $D$ is simple normal crossing in the following
sense: if we write $D=\sum a_{i}D_{i}$ with $D_{i}$ prime divisors
and $a_{i}\ne0$, then
\begin{enumerate}
\item for every $i$ and every $k$-point $x\in D_{i}$ where $X$ is $\cO_{K}$-smooth,
the completion of $D_{i}$ is irreducible, and
\item for every $k$-point $x\in X$ where $X$ is $\cO_{K}$-smooth, there
exists a regular system of parameters $x_{0}=\varpi,x_{1},\dots,x_{d}$
on a neighborhood $U$ of $x$ with $\varpi$ a uniformizer of $K$
such that the support of $D$ is defined by a product $\prod_{j=1}^{m}x_{i_{j}}$
$(0\le i_{1}<\cdots<i_{m}\le d)$. 
\end{enumerate}
We then rewrite $D$ as

\[
D=\sum_{h=1}^{l}a_{h}A_{h}+\sum_{i=1}^{m}b_{i}B_{i}+\sum_{j=1}^{n}c_{j}C_{j}
\]
such that
\begin{enumerate}
\item for every $h$, $a_{h}\ne0$, $A_{h}\subset X_{k}$ and $X$ is $\cO_{K}$-smooth
at the generic point of $A_{h}$,
\item for every $i$, $b_{i}\ne0$, $B_{i}\subset X_{k}$ and $X$ is \emph{not}
$\cO_{K}$-smooth at the generic point of $B_{i}$, and
\item for every $j$, $c_{j}\ne0$ and $C_{j}$ dominates $\Spec\cO_{K}$. 
\end{enumerate}
Such a decomposition of $D$ is unique. For each $h$, we let $A_{h}^{\circ}$
be the locus in $A_{h}$ where $X$ is $\cO_{K}$-smooth. For each
subset $J\subset\{1,\dots,n\}$, we put
\[
C_{J}^{\circ}:=\left(\bigcap_{j\in J}C_{j}\right)\setminus\left(\bigcup_{j\in\{1,\dots,n\}\setminus J}C_{j}\right).
\]
Let $X_{\sm}$ be the $\cO_{K}$-smooth locus of $X$.
\begin{prop}
The stringy point count $\sharp_{\st}(X,D)_{C}$ is finite if and
only if $c_{j}<1$ for every $j$ such that $(C_{j}\cap C\cap X_{\sm})(k)\ne\emptyset$.
If these equivalent conditions hold, we have
\[
\sharp_{\st}(X,D)_{C}=\sum_{h=1}^{l}q^{a_{h}}\sum_{J\subset\{1,\dots,n\}}\sharp(C\cap A_{h}^{\circ}\cap C_{J}^{\circ})(k)\prod_{j\in J}\frac{q-1}{q^{1-c_{j}}-1}.
\]
\end{prop}
\begin{proof}
The proof presented here follows the one of \cite[Proposition 3.4]{MR2098399}.
There is no $\cO_{K}$-point passing through a $k$-point where $X$
is not $\cO_{K}$-smooth. Therefore 
\[
\sharp_{\st}(X,D)_{C}=\sum_{x\in(C\cap X_{\sm})(k)}\sharp_{\st}(X,D)_{x}.
\]
We fix $x\in(C\cap X_{\sm})(k)$ and suppose that $C_{1},\dots,C_{\lambda}$
are those prime divisors among $C_{1},\dots,C_{n}$ containing $x$.
Let $a$ be $a_{h}$ if $A_{h}$ contains $x$, and zero if none of
$A_{h}$ contains $x$. To show the proposition, it suffices to show:
\begin{claim*}
If $c_{j}\ge1$ for some $j$ with $1\le j\le\lambda$, $\sharp_{\st}(X,D)_{x}=\infty$,
and otherwise,
\[
\sharp_{\st}(X,D)_{x}=q^{a}\cdot\prod_{j=1}^{\lambda}\frac{q-1}{q^{1-c_{j}}-1}.
\]

\end{claim*}
To see this, using local coordinates $x_{0}=\varpi,x_{1},\dots,x_{d}$,
we suppose that $A_{h}$ containing $x$ (if any) is defined by $x_{0}$
and for $1\le i\le\lambda$, $C_{i}$ is defined by $x_{i}$. Then,
for an integer $r>0$, 
\begin{align*}
\sharp_{\st}(X,D)_{x} & =q^{d}\int_{X(\cO_{K})_{x}}|\varpi^{-ra}x_{1}^{-rc_{1}}\cdots x_{\lambda}^{-rc_{\lambda}}|^{1/r}dx_{1}\wedge\cdots\wedge dx_{d}\\
 & =q^{a}\cdot\prod_{j=1}^{d}q\int_{\fm_{K}}|x|^{-c_{j}}\,dx,
\end{align*}
with $c_{j}:=0$ for $j>\lambda$. For any $c\in\RR$, we have
\begin{align*}
\int_{\fm_{K}}|x|^{-c}\,dx & =\sum_{i=1}^{\infty}q^{ic}\cdot\mu_{K}(\fm_{K}^{i}\setminus\fm_{K}^{i+1})\\
 & =\sum_{i=1}^{\infty}q^{ic}\cdot(q^{-i}-q^{-i-1})\\
 & =\begin{cases}
q^{-1}\cdot\frac{q-1}{q^{1-c}-1} & (c<1)\\
\infty & (c\ge1).
\end{cases}
\end{align*}
This shows the claim and the proposition.
\end{proof}

\section{Group actions\label{sec:Group-actions}}

In this section, we consider a $G$-cover of varieties $V\to X=V/G$
with $G$ a finite group and show a correspondence of $\cO_{K}$-points
of $X$ and equivariant $\cO_{M}$-points with $M$ $G$-étale $K$-algebras.
From now on, $G$ denotes a finite group.

\subsection{\label{sub:G-alg}}
\begin{defn}
A \emph{$G$-étale $K$-algebra }means a finite $K$-algebra $M$
of degree $\sharp G$ endowed with a (right) $G$-action such that
the subset of $G$-invariant elements, $M^{G}$, is identical to $K$.
An \emph{isomorphism }$M\to N$ of $G$-étale $K$-algebras is a $K$-algebra
isomorphism compatible with the given $G$-actions on $M$ and $N$.
We denote the set of representatives of isomorphism classes of $G$-étale
$K$-algebras by $\GEt$. We denote the automorphism group of a $G$-étale
$K$-algebra $M$ by $\Aut^{G}(M/K)$.
\end{defn}
Let $M\in\GEt$. There exists a field extension $L$ of $K$ such
that $M$ is isomorphic to the product $L^{c}$ of $c$ copies of
$L$ as an $K$-algebra for some positive integer $c$. Geometrically
we can write 
\[
\Spec M=\overset{c}{\overbrace{\Spec L\sqcup\cdots\sqcup\Spec L}}.
\]
Let $H\subset G$ be the stabilizer of one connected component of
$\Spec M$. The subgroup $H$ depends on the choice of the chosen
component, but is unique up to conjugation in $G$. The automorphism
group $\Aut^{G}(M/K)$ is isomorphic to $C_{G}(H)^{\op}$, the opposite
group of the centralizer $H$ (for instance, see \cite{Yasuda:2014fk2}).
\begin{rem}
$G$-étale $K$-algebras $M$ correspond to $G$-conjugacy classes
of continuous homomorphisms $\rho:\Gal(K^{\sep}/K)\to G$ with $K^{\sep}$
a separable closure of $K$. The stabilizer $H$ of a connected component
of $\Spec M$ coincides with the image of the corresponding map $\rho$
up to conjugation. Since the Galois group of a finite Galois extension
of a local field is always solvable (see \cite[page 68]{MR554237}),
if $G$ is not solvable, then every $G$-étale $K$-algebra $M$ is
not a field. 
\end{rem}

\subsection{\label{sub:action}}

Let $V$ be an $\cO_{K}$-variety endowed with a faithful $G$-action
and let $X:=V/G$ be the quotient variety. For $M\in\GEt$, we let
$V(\cO_{M})^{G}$ be the set of $G$-equivariant $\cO_{M}$-points
of $V$, that is, $G$-equivariant $\cO_{K}$-morphisms $\Spec\cO_{M}\to V$.
A $G$-equivariant $\cO_{M}$-point $\Spec\cO_{M}\to V$ induces a
natural morphism between the quotients of the source and target: $\Spec\cO_{K}\to X$.
This defines a map
\[
V(\cO_{M})^{G}\to X(\cO_{K}).
\]

Let $H\subset G$ be the stabilizer of a connected component of $\Spec M$
as above. The restriction of the $G$-action on $V$ to $C_{G}(H)$
induces a (left) $C_{G}(H)$-action on $V(\cO_{M})^{G}$; $g\in C_{G}(H)$
sends a $G$-equivariant point $\alpha:\Spec\cO_{M}\to V$ to $g\circ\alpha$.
Identifying $\Aut^{G}(M/K)$ with $C_{G}(H)^{\op}$, this action is
identical to the left $\Aut^{G}(M/K)^{\op}$-action corresponding
to the right $\Aut^{G}(M/K)$-action induced from the $\Aut^{G}(M/K)$-action
on $\Spec\cO_{M}$. The map $V(\cO_{M})^{G}\to X(\cO_{K})$ factors
through $V(\cO_{M})^{G}/C_{G}(H)$.
\begin{defn}
Let $X(\cO_{K})^{\natural}$ be the set of $\cO_{K}$-points $\Spec\cO_{K}\to X$
sending the generic point into the unramified locus of $V\to X$ in
$X$, and let $V(\cO_{M})^{G,\natural}$ be the set of $G$-equivariant
$\cO_{M}$-points sending the generic points into the unramified locus
of the same morphism in $V$. 
\end{defn}
The map $V(\cO_{M})^{G}/C_{G}(H)\to X(\cO_{K})$ restricts to $V(\cO_{M})^{G,\natural}/C_{G}(H)\to X(\cO_{K})^{\natural}$. 
\begin{prop}
\label{prop:general point corr}The map
\[
\bigsqcup_{M\in\GEt}V(\cO_{M})^{G,\natural}/C_{G}(H)\to X(\cO_{K})^{\natural}
\]
is bijective. \end{prop}
\begin{proof}
We first show the injectivity. Let $\alpha:\Spec\cO_{M}\to V$ be
a $G$-equivariant point in $V(\cO_{M})^{\natural}$ and $\beta:\Spec\cO_{K}\to X$
its image in $X(\cO_{K})^{\natural}$. Then the class of $\alpha$
modulo the $C_{G}(H)$-action is reconstructed from $\beta$ as the
normalization of $\Spec\cO_{K}\times_{\beta,X,\psi}V$ with $\psi$
the quotient morphism $V\to X$. Indeed the induced morphism $\Spec M\to\Spec K\times_{\beta,X,\psi}V$
is a morphism of étale $G$-torsors over $\Spec K$, hence is an isomorphism.
This shows the injectivity.

Given a point $\beta:\Spec\cO_{K}\to X$ in $X(\cO_{K})^{\natural}$,
the normalization of $\Spec\cO_{K}\times_{\beta,X,\psi}V$ is the
spectrum of $\cO_{M}$ for a $G$-étale $K$-algebra $M$. This shows
the surjectivity.
\end{proof}

\section{Untwisting\label{sec:Untwisting}}

The key for the formulation and the proof of our main results is \emph{untwisting},
which makes a correspondence between a set of equivariant $\cO_{M}$-points
of a $G$-variety $V$ and a set of $\cO_{K}$-points of another variety
$V^{|M|}$ constructed by twisting $V$ somehow. At the cost of the
twist of $V$, the study of equivariant points (twisted arcs) reduces
to the one of ordinary points (ordinary arcs).

\subsection{\label{sub:untwist1}}

We fix an $\cO_{K}$-linear faithful $G$-action on an affine space
\[
\AA_{\cO_{K},\bx}^{n}:=\Spec\cO_{K}[\bx]=\Spec\cO_{K}[x_{1},\dots,x_{n}].
\]
Let $\cO_{K}[\bx]_{1}$ be the linear part of $\cO_{K}[\bx]$. We
introduce a notion playing the central role in the untwisting technique.
\begin{defn}
We define the \emph{tuning module $\Xi_{M}$ }by 
\[
\Xi_{M}:=\{\phi\in\Hom_{\cO_{K}}(\cO_{K}[\bx]_{1},\cO_{M})\mid\forall g\in G,\,\phi\circ g=g\circ\phi\}.
\]

\end{defn}
The module $\Xi_{M}$ is actually identified with $\AA_{\cO_{K},\bx}^{n}(\cO_{M})^{G}$
by the map 
\[
\AA_{\cO_{K},\bx}^{n}(\cO_{M})^{G}\to\Xi_{M},\,\gamma\mapsto\gamma^{*}|_{\cO_{K}[\bx]_{1}}.
\]
It turns out that the tuning module $\Xi_{M}$ is a free $\cO_{K}$-submodule
of $\Hom_{\cO_{K}}(\cO_{K}[\bx]_{1},\cO_{M})=\AA_{\cO_{K},\bx}^{n}(\cO_{M})$
of rank $n$ \cite{Yasuda:2013fk,Wood-Yasuda-I}. 
\begin{rem}
\label{rem: left or right}Our definition of the tuning module follows
the one in \cite{Yasuda:2014fk2}. Noting that throughout the literature
\cite{Wood-Yasuda-I,Yasuda:2013fk,Yasuda:2014fk2}, a finite group
always acts on an affine space $\AA_{\cO_{K},\bx}^{n}$ from left,
if we think that the tuning module is associated to $M$ and the $G$-variety
$\AA_{\cO_{K},\bx}^{n}$ (rather than its coordinate ring) and if
$\AA_{\cO_{K},\bx}^{n}(\cO_{M})$ is identified with $\cO_{M}^{\oplus n}$,
then our definition of $\Xi_{M}$ coincides with the tuning submodule
considered in \cite{Yasuda:2013fk,Wood-Yasuda-I}. On the other hand,
the coordinate ring $\cO_{K}[\bx]$ has a right or left action, depending
on the paper, and one has to be careful about the caused notational
difference.
\end{rem}
We fix a basis $\phi_{1},\dots,\phi_{n}\in\Xi_{M}$ and let $y_{1},\dots,y_{n}$
be its dual basis so that 
\[
\Hom_{\cO_{K}}(\Xi_{M},\cO_{K})=\bigoplus_{j}\cO_{K}\cdot y_{j}\text{ and }\Hom_{\cO_{K}}(\Xi_{M},\cO_{M})=\bigoplus_{j}\cO_{M}\cdot y_{j}.
\]
We think of these modules as the linear parts of the polynomial rings
$\cO_{K}[\by]=\cO_{K}[y_{1},\dots,y_{n}]$ and $\cO_{M}[\by]=\cO_{M}[y_{1},\dots,y_{n}]$.
We put
\[
\AA_{\cO_{K},\by}^{n}:=\Spec\cO_{K}[\by]\text{ and }\AA_{\cO_{M},\by}^{n}:=\Spec\cO_{M}[\by].
\]

\begin{defn}
\label{def:u}We define an $\cO_{K}$-algebra morphism $u^{*}:\cO_{K}[\bx]\to\cO_{M}[\by]$
by
\[
u^{*}(x_{i})=\sum_{j=1}^{n}\phi_{j}(x_{i})y_{j}
\]
and the corresponding morphism of schemes
\[
u:\AA_{\cO_{M},\by}^{n}\to\AA_{\cO_{K},\bx}^{n}.
\]

\end{defn}
The linear part of $u^{*}$ is identical to the canonical map
\begin{align*}
\cO_{K}[\bx]_{1} & \to\Hom_{\cO_{K}}(\Xi_{M},\cO_{M})\\
f & \mapsto(\phi\mapsto\phi(f)),
\end{align*}
which gives an intrinsic description of $u$. This is useful in order
to show that some derived maps are equivariant.
\begin{lem}
\label{lem:u etale}The restriction of $u$, $\AA_{M,\by}^{n}\to\AA_{K,\bx}^{n}$,
is étale.\end{lem}
\begin{proof}
The chosen basis $\phi_{1},\dots,\phi_{n}$ of $\Xi_{M}$ is an $M$-basis
of the free $M$-module 
\[
\Hom_{\cO_{K}}(\cO_{K}[\bx]_{1},M)=\Hom_{M}(M[\bx]_{1},M)
\]
and its dual basis $y_{1},\dots,y_{n}$ is naturally regarded as a
basis of 
\[
\Hom_{M}(\Hom_{M}(M[\bx]_{1},M),M).
\]
Therefore the map $K[\bx]_{1}\to M[\by]_{1}$ corresponding to the
morphism of the lemma is identified with 
\[
K[\bx]_{1}\to M[\bx]_{1}\to\Hom_{M}(\Hom_{M}(M[\bx]_{1},M),M),
\]
the composition of the scalar extension and the canonical morphism
to the double dual space. This proves the lemma.
\end{proof}

\subsection{\label{sub:two actions}}
\begin{defn}
The given $G$-action on $\Spec\cO_{M}$ and the trivial $G$-action
on $\AA_{\cO_{K,\by}}^{n}$ defines a $G$-action on the fiber product
$\AA_{\cO_{M},\by}^{n}=\Spec\cO_{M}\times_{\Spec\cO_{K}}\AA_{\cO_{K},\by}^{n}$
over $\cO_{K}$: we call it the \emph{Galois $G$-action.}\end{defn}
\begin{lem}
The map $u$ is equivariant with respect to the given $G$-action
on $\AA_{\cO_{K},\bx}^{n}$ and the Galois $G$-action on $\AA_{\cO_{M},\by}^{n}$. \end{lem}
\begin{proof}
For $1\le i\le n$ and $g\in G$, we have 
\begin{align*}
u^{*}(x_{i}g) & =\sum_{j=1}^{n}\phi_{j}(x_{i}g)y_{j}.
\end{align*}
From the definition of $\Xi_{M}$, we have
\[
u^{*}(x_{i}g)=\sum_{j=1}^{n}\phi_{j}(x_{i}g)y_{j}=\sum_{j=1}^{n}(\phi_{j}(x_{i})g)y_{j}=u^{*}(x_{i})g,
\]
and the lemma follows. \end{proof}
\begin{prop}
We have 
\[
u^{*}\left(\cO_{K}[\bx]^{G}\right)\subset\cO_{K}[\by].
\]
\end{prop}
\begin{proof}
Note that $\cO_{K}[\by]$ is identical to the invariant subring $\cO_{M}[\by]^{G}$
with respect to the Galois $G$-action as above. Since $u^{*}$ is
$G$-equivariant, all $G$-invariant elements of $\cO_{K}[\bx]$ map
into $\cO_{K}[\by]$. 
\end{proof}
Let $\Psi:\AA_{\cO_{K},\bx}^{n}\to\AA_{\cO_{K},\bx}^{n}/G$ be the
quotient morphism and $\Psi^{|M|}:\AA_{\cO_{K},\by}^{n}\to\AA_{\cO_{K},\bx}^{n}/G$
the morphism corresponding to $u^{*}:\cO_{K}[\bx]^{G}\to\cO_{K}[\by]$.
We obtain the following commutative diagram:
\[
\xymatrix{ & \AA_{\cO_{M},\by}^{n}\ar[dl]_{u}\ar[dr]^{\otimes_{\cO_{K}}\cO_{M}}\\
\AA_{\cO_{K},\bx}^{n}\ar[dr]_{\Psi} &  & \AA_{\cO_{K},\by}^{n}\ar[dl]^{\Psi^{|M|}}\\
 & \AA_{\cO_{K},\bx}^{n}/G
}
\]

\subsection{\label{sub:point corresp}}
\begin{defn}
The action of $C_{G}(H)=\Aut(M)^{\op}$ on $\Xi_{M}$ induces $C_{G}(H)$-action
on $\cO_{K}[\by]_{1}=\Hom_{\cO_{K}}(\Xi_{M},\cO_{K})$\emph{ }and\emph{
}$\cO_{M}[\by]_{1}=\Hom_{\cO_{K}}(\Xi_{M},\cO_{M})$, which are $\cO_{K}$-linear
and $\cO_{M}$-linear respectively. In turn, they induce $C_{G}(H)$-actions
on $\AA_{\cO_{K},\by}^{n}$ and $\AA_{\cO_{M},\by}^{n}$: we call
all these actions the \emph{non-Galois action}s. \end{defn}
\begin{lem}

\begin{enumerate}
\item The map $u$ is equivariant with respect to the restriction of the
given $G$-action on $\AA_{\cO_{K},\bx}^{n}$ to $C_{G}(H)$ and the
non-Galois $C_{G}(H)$-action on $\AA_{\cO_{M},\by}^{n}$. 
\item The map $\AA_{\cO_{M},\by}^{n}\to\AA_{\cO_{K},\by}^{n}$ given by
the scalar extension is equivariant with respect to the non-Galois
$C_{G}(H)$-actions.
\end{enumerate}
\end{lem}
\begin{proof}
The natural map 
\[
\cO_{K}[\bx]_{1}\to\cO_{M}[\by]_{1}=\Hom_{\cO_{K}}(\Xi_{M},\cO_{M})
\]
is clearly $C_{G}(H)$-equivariant, and the first assertion of the
lemma follows. The second assertion is trivial.\end{proof}
\begin{defn}
Let $\AA_{\cO_{M},\by}^{n}(\cO_{M})^{G}$ be the set of $G$-equivariant
$\cO_{M}$-morphisms $\Spec\cO_{M}\to\AA_{\cO_{M},\by}^{n}$. The
sets $\AA_{\cO_{M},\by}^{n}(\cO_{M})^{G}$ and $\AA_{\cO_{K},\by}^{n}(\cO_{K})$
both have $C_{G}(H)$-actions by $g\cdot\gamma=g\circ\gamma$: we
call these again \emph{non-Galois}. \end{defn}
\begin{lem}
The map obtained by the scalar extension,
\[
\AA_{\cO_{K},\by}^{n}(\cO_{K})\to\AA_{\cO_{M},\by}^{n}(\cO_{M})^{G},
\]
is bijective and $C_{G}(H)$-equivariant with respect to the non-Galois
actions. \end{lem}
\begin{proof}
A $G$-equivariant (with respect to the Galois action on $\AA_{\cO_{M},\by}^{n}$)
$\cO_{M}$-point $\gamma$ is determined by $\gamma(y_{j})$, $1\le j\le n$
which must lie in $\cO_{K}$. This shows the bijectivity. That it
is equivariant follows from that $\AA_{\cO_{M},\by}^{n}\to\AA_{\cO_{K},\by}^{n}$
is equivariant.
\end{proof}
If $\gamma\in\AA_{\cO_{M},\by}^{n}(\cO_{M})^{G}$, then $u\circ\gamma\in\AA_{\cO_{K},\bx}^{n}(\cO_{M})^{G}$.
It defines a map
\[
\alpha:\AA_{\cO_{M},\by}^{n}(\cO_{M})^{G}\to\AA_{\cO_{K},\bx}^{n}(\cO_{M})^{G},\,\gamma\mapsto u\circ\gamma.
\]

\begin{lem}
The map $\alpha$ is a $C_{G}(H)$-equivariant bijection with respect
to the non-Galois action on $\AA_{\cO_{M},\by}^{n}(\cO_{M})^{G}$.\end{lem}
\begin{proof}
The map is clearly $C_{G}(H)$-equivariant. As for the bijectivity,
the point is that both sets $\AA_{\cO_{M},\by}^{n}(\cO_{M})^{G}$
and $\AA_{\cO_{K},\bx}^{n}(\cO_{M})^{G}$ are naturally identified
with $\Xi_{M}$. A point $\gamma\in\AA_{\cO_{M},\by}^{n}(\cO_{M})$
is $G$-equivariant if and only if $\gamma^{*}(y_{i})\in\cO_{K}=(\cO_{M})^{G}$.
Recalling that $y_{1},\dots,y_{n}$ are the dual basis of $\phi_{1},\dots,\phi_{n}$,
we identify $\AA_{\cO_{M},\by}^{n}(\cO_{M})^{G}$ with $\bigoplus_{i=1}^{n}\cO_{K}\cdot\phi_{i}=\Xi_{M}$.
By Definition \ref{def:u}, $(u(\phi_{i}))^{*}$ sends $x_{i'}$ to
\[
\sum_{j=1}^{n}\phi_{j}(x_{i'})\phi_{i}(y_{j})=\phi_{i}(x_{i'}).
\]
Namely the restriction of $(u(\phi_{i}))^{*}$ to the linear part
$\cO_{K}[\bx]_{1}$ is $\phi_{i}$. Thus, with the obvious identification
$\AA_{\cO_{K},\bx}^{n}(\cO_{M})^{G}=\Xi_{M}$, the map $\alpha$ corresponds
to the identity map of $\Xi_{M}$. In particular, $\alpha$ is bijective. 
\end{proof}
We have obtained two one-to-one correspondences which are $C_{G}(H)$-equivariant,

\[
\xymatrix{ & \AA_{\cO_{M},\by}^{n}(\cO_{M})^{G}\ar@{<->}[dl]\ar@{<->}[dr]\\
\AA_{\cO_{K},\bx}^{n}(\cO_{M})^{G} &  & \AA_{\cO_{K},\by}^{n}(\cO_{K}),
}
\]
which induce one-to-one correspondences
\[
\xymatrix{ & \frac{\AA_{\cO_{M},\by}^{n}(\cO_{M})^{G}}{C_{G}(H)}\ar@{<->}[dl]\ar@{<->}[dr]\\
\frac{\AA_{\cO_{K},\bx}^{n}(\cO_{M})^{G}}{C_{G}(H)} &  & \frac{\AA_{\cO_{K},\by}^{n}(\cO_{K})}{C_{G}(H)}.
}
\]

\begin{defn}
Let $\AA_{\cO_{K},\by}^{n}(\cO_{K})^{\natural}$ (resp. $\AA_{\cO_{M},\by}^{n}(\cO_{M})^{G,\natural}$)
be the set of $\cO_{K}$-points (resp. $G$-equivariant $\cO_{M}$-points)
sending the generic point(s) into the locus where $\AA_{\cO_{K},\by}^{n}\to\AA_{\cO_{K},\bx}^{n}/G$
(resp. $\AA_{\cO_{M},\by}^{n}\to\AA_{\cO_{K},\bx}^{n}/G$) is étale. 
\end{defn}
From Lemma \ref{lem:u etale}, the above correspondences restrict
to the correspondences
\[
\xymatrix{ & \frac{\AA_{\cO_{M},\by}^{n}(\cO_{M})^{G,\natural}}{C_{G}(H)}\ar@{<->}[dl]\ar@{<->}[dr]\\
\frac{\AA_{\cO_{K},\bx}^{n}(\cO_{M})^{G,\natural}}{C_{G}(H)} &  & \frac{\AA_{\cO_{K},\by}^{n}(\cO_{K})^{\natural}}{C_{G}(H)}.
}
\]
From Proposition \ref{prop:general point corr}, we obtain:
\begin{prop}
\label{prop:bij1}The natural map
\[
\bigsqcup_{M\in\GEt}\frac{\AA_{\cO_{K},\by}^{n}(\cO_{K})^{\natural}}{C_{G}(H)}\to\left(\AA_{\cO_{K},\bx}^{n}/G\right)(\cO_{K})^{\natural},
\]
induced by the morphisms $\Psi^{|M|}:\AA_{\cO_{K},\by}^{n}\to\AA_{\cO_{K},\bx}^{n}/G$,
is bijective. Here $H$, $\AA_{\cO_{K},\by}^{n}$ and $\Psi^{|M|}$
vary, depending on $M$.
\end{prop}

\subsection{\label{sub:G-var V}}

Let $V\subset\AA_{\cO_{K},\bx}^{n}$ be a normal closed $\cO_{K}$-subvariety
stable under the $G$-action such that the induced $G$-action on
$V$ is faithful. Let $X$ be the quotient variety $V/G$ and $\overline{X}$
be the image of $V$ in $\AA_{\cO_{K},\bx}^{n}/G$. The canonical
morphism $X\to\overline{X}$ is finite and birational. If $U\subset X$
and $\overline{U}\subset\overline{X}$ denote the loci where $V\to X$
and $V\to\overline{X}$ are unramified, then the morphism $X\to\overline{X}$
induces an isomorphism $U\to\overline{U}$.
\begin{defn}
For $M\in\GEt$, we define $V^{|M|}$ to be the preimage of $\overline{X}$
in $\AA_{\cO_{K},\by}^{n}$ and $V^{\left\langle M\right\rangle }$
to be the preimage of $\overline{X}$ in $\AA_{\cO_{M},\by}^{n}$,
both of which are given the reduced scheme structures. 
\end{defn}
Let $V^{\left\langle M\right\rangle ,\nu}$ and $V^{|M|,\nu}$ be
the normalizations of $V^{\left\langle M\right\rangle }$ and $V^{|M|}$
respectively. The normalization morphisms $V^{\left\langle M\right\rangle ,\nu}\to V^{\left\langle M\right\rangle }$
and $V^{|M|,\nu}\to V^{\left|M\right|}$ are isomorphisms over $V^{\left\langle M\right\rangle }\otimes M$
and $V^{|M|}\otimes K$ respectively. We have got the following diagram:
\[
\xymatrix{ & V^{\left\langle M\right\rangle ,\nu}\ar[d]\ar[dr]\\
 & V^{\left\langle M\right\rangle }\ar[dl]\ar[dr] & V^{|M|,\nu}\ar[d]\ar[ddl]\\
V\ar[ddr]_{\psi}\ar[dr] &  & V^{|M|}\ar[ddl]^{\psi^{|M|}}\\
 & X\ar[d]\\
 & \overline{X}
}
\]
Here we name morphisms as in the diagram. All morphisms here are generically
étale. Let $\overline{X}(\cO_{K})^{\natural}$ be the set of $\cO_{K}$-points
of $\overline{X}$ sending the generic point into the unramified locus
of $V\to\overline{X}$ (equivalently of $V^{|M|}\to\overline{X}$)
and $V^{|M|}(\cO_{K})^{\natural}$ (resp. $V^{|M|,\nu}(\cO_{K})^{\natural}$)
be the set of $\cO_{K}$-points of $V^{|M|}$ (resp. $V^{|M|,\nu}$)
sending the generic point into the unramified locus of $V^{|M|}\to\overline{X}$
(resp. $V^{|M|,\nu}\to\overline{X}$).
\begin{prop}
\label{prop:bij2}The natural maps
\begin{gather*}
\bigsqcup_{M\in\GEt}\frac{V^{|M|}(\cO_{K})^{\natural}}{C_{G}(H)}\to\overline{X}(\cO_{K})^{\natural}\text{ and}\\
\bigsqcup_{M\in\GEt}\frac{V^{|M|,\nu}(\cO_{K})^{\natural}}{C_{G}(H)}\to X(\cO_{K})^{\natural}
\end{gather*}
are bijective. Here the subgroup $H\subset G$ varies, depending on
$M$.\end{prop}
\begin{proof}
The first map is a restriction of the map in Proposition \ref{prop:bij1}
and easily seen to be bijective. Since $V^{|M|,\nu}\otimes K\to V^{|M|}\otimes K$
and $U\to\overline{U}$ are isomorphisms, we have natural bijections
$V^{|M|,\nu}(\cO_{K})^{\natural}\to V^{|M|}(\cO_{K})^{\natural}$
and $X(\cO_{K})^{\natural}\to\overline{X}(\cO_{K})^{\natural}$. It
follows that the second map of the proposition is also bijective. 
\end{proof}

\section{Main results\label{sec:Main-results}}

Using the untwisting, we introduce the notion of \emph{$G$-stringy
point counts }for $G$-log pairs and prove our main results.

\subsection{\label{sub:G-pair}}
\begin{defn}
A \emph{$G$-log pair} is a log pair $(V,E)$ with a faithful $G$-action
on $V$ such that $E$ is $G$-stable, that is, for every $g\in G$,
$g_{*}E=E$. 
\end{defn}
Let $(V,E)$ be a $G$-log pair and let $X:=V/G$ be the quotient
scheme, and let $\pi:V\to X$ be the quotient morphism. 
\begin{lem}
There exists a unique $\QQ$-divisor $D$ such that $(X,D)$ is a
log pair and the induced morphism $(V,E)\to(X,D)$ is crepant.\end{lem}
\begin{proof}
Let $K_{V/X}$ be the ramification divisor of $\pi$, defined so that
the equality of subsheaves of $\omega_{V}$, 
\[
\omega_{V}(-K_{V/X})=\pi^{*}\omega_{X},
\]
holds in codimension one. We put $D:=\frac{1}{\sharp G}\pi_{*}(E-K_{V/X})$.
The pull-back $\pi^{*}(K_{X}+D)$ is defined at least in codimension
one and coincides with $K_{V}+E$. Since the pull-back map $\pi^{*}$
gives a one-to-one correspondence of $\QQ$-Cartier divisors on $X$
and $G$-stable $\QQ$-Cartier divisors on $V$, we conclude that
$K_{X}+D$ is $\QQ$-Cartier. Thus $(X,D)$ is a log variety such
that $(V,E)\to(X,D)$ is crepant. The uniqueness of $D$ is obvious.
\end{proof}

\subsection{\label{sub:embed}}

In this subsection, we suppose that $(V,E)$ is a $G$-log pair with
$V$ affine. Then there exist an affine variety $\AA_{\cO_{K}}^{n}$
with a faithful $\cO_{K}$-linear $G$-action and a $G$-equivariant
closed embedding $V\hookrightarrow\AA_{\cO_{K}}^{n}$. Indeed, if
the coordinate ring $\cO_{V}$ of $V$ is generated by $f_{1},\dots,f_{l}$
as an $\cO_{K}$-algebra, then we let $S$ be the union of the $G$-orbits
of these generators $f_{i}$ and consider the polynomial ring $\cO_{K}[x_{s}\mid s\in S]$
having indeterminants corresponding to elements of $S$. Now $G$
acts on this polynomial ring faithfully and $\cO_{K}$-linearly and
the natural map
\[
\cO_{K}[x_{s}\mid s\in S]\to\cO_{V},\,x_{s}\mapsto s
\]
defines a desired embedding $V\hookrightarrow\AA_{\cO_{K}}^{\sharp S}$. 

We fix such an embedding $V\subset\AA_{\cO_{K},\bx}^{n}$ and follow
the notation in section \ref{sec:Untwisting}. In particular, for
each $M\in\GEt$, we obtain the diagram in section \ref{sub:G-var V}.
We define the $G$-log structure $(V^{\left\langle M\right\rangle ,\nu},E^{\left\langle M\right\rangle ,\nu})$
and log structures $(V^{|M|,\nu},E^{|M|,\nu})$ and $(X,D)$ so that
all the solid arrows in the diagram
\[
\xymatrix{ & (V^{\left\langle M\right\rangle ,\nu},E^{\left\langle M\right\rangle ,\nu})\ar[dl]\ar[dr]\\
(V,E)\ar[dr] &  & (V^{|M|,\nu},E^{|M|,\nu})\ar@{-->}[dl]\\
 & (X,D),
}
\]
are crepant. Then the dashed arrow is also crepant, which shows $(V^{|M|,\nu},E^{|M|,\nu})$
is a $C_{G}(H)$-log pair. For a $G$-stable constructible subset
$C\subset V_{k}$ and $M\in\GEt$, we define $C^{\left\langle M\right\rangle ,\nu}\subset V_{k}^{\left\langle M\right\rangle ,\nu}$
to be the preimage of $C$ and $C^{|M|,\nu}\subset V_{k}^{|M|,\nu}$
to be its image. 
\begin{defn}
For $M\in\GEt$, we define the \emph{$M$-stringy point count }of
$(V,E)$ along $C$ by
\[
\sharp_{\st}^{M}(V,E)_{C}:=\frac{\sharp_{\st}(V^{|M|,\nu},E^{|M|,\nu})_{C^{|M|,\nu}}}{\sharp C_{G}(H)}
\]
and the $G$\emph{-stringy point count} along $C$ by
\[
\sharp_{\st}^{G}(V,E)_{C}:=\sum_{M\in\GEt}\sharp_{\st}^{M}(V,E)_{C}.
\]
Again we omit the subscript $C$ when $C=V_{k}$. \end{defn}
\begin{thm}
\label{thm:main affine}For a $G$-log pair $(V,E)$ with $V$ affine,
let $(X,D)$ be as above. Let $C\subset V_{k}$ be a $G$-stable subset
and $\overline{C}\subset X_{k}$ its image. Then we have
\[
\sharp_{\st}^{G}(V,E)_{C}=\sharp_{\st}(X,D)_{\overline{C}}.
\]
In particular, we have
\[
\sharp_{\st}^{G}(V,E)=\sharp_{\st}(X,D).
\]
\end{thm}
\begin{proof}
Let $d$ be the dimension of the $\cO_{K}$-variety $V$. From Lemma
\ref{lem:measure zero 2}, Theorem \ref{thm:change-vars} and Propositions
\ref{prop:bij2}, we have
\begin{align*}
\sharp_{\st}(X,D) & _{\overline{C}}=q^{d}\cdot\mu_{X,D}(X(\cO_{K})_{\overline{C}})\\
 & =\sum_{M\in\GEt}\frac{q^{d}\cdot\mu_{V^{|M|,\nu},E^{|M|,\nu}}(V^{|M|,\nu}(\cO_{K})_{C^{|M|,\nu}})}{\sharp C_{G}(H)}\\
 & =\sum_{M\in\GEt}\frac{\sharp_{\st}(V^{|M|,\nu},E^{|M|,\nu})_{C^{|M|,\nu}}}{\sharp C_{G}(H)}\\
 & =\sum_{M\in\GEt}\sharp_{\st}^{M}(V,E)_{C}\\
 & =\sharp_{\st}^{G}(V,E)_{C}.
\end{align*}
\end{proof}
\begin{cor}
\label{cor:main non-log}Let $V$ be a normal $\QQ$-Gorenstein affine
$\cO_{K}$-variety endowed with a faithful $G$-action and $X:=V/G$
its quotient scheme. Suppose that the quotient morphism $V\to X$
is étale in codimension one. Then, for a $G$-stable constructible
subset $C\subset V_{k}$, we have 
\[
\sharp_{\st}^{G}(V)_{C}=\sharp_{\st}(X)_{\overline{C}}.
\]
In particular,
\[
\sharp_{\st}^{G}V=\sharp_{\st}X.
\]
\end{cor}
\begin{proof}
From the assumption, the morphism 
\[
V=(V,0)\to X=(X,0)
\]
is crepant. Therefore the corollary is a special case of the preceding
theorem.
\end{proof}

\subsection{\label{sub:main general}}

We now consider an arbitrary $G$-log pair $(V,E)$ ($V$ is not necessarily
affine but quasi-projective from our definition of varieties in section
\ref{sec:Convention-and-notation}). Let us take an affine open cover
$V=\bigcup_{i}V_{i}$ such that each $V_{i}$ is $G$-stable. For
each $M\in\GEt$, let $\mu_{V,E,i}^{M}$ be the measure on $V_{i}(\cO_{M})^{G,\circ}$
corresponding to 
\[
\mu_{V_{i}^{|M|,\nu},(E|_{V_{i}})^{|M|,\nu}}
\]
through the correspondence $V_{i}(\cO_{M})^{G,\circ}\leftrightarrow V_{i}^{|M|}(\cO_{K})^{\circ}$.
The argument of the proof of Theorem \ref{thm:main affine} shows
that the measures $\mu_{V,E,i}^{M}$ and $\mu_{V,E,j}^{M}$ coincide
on $(V_{i}\cap V_{j})(\cO_{M})^{G,\circ}$, and we obtain a measure
on $V(\cO_{M})^{G,\circ}$, and one on $V(\cO_{M})^{G}$ by extending
it so that all subsets of $V(\cO_{M})^{G}\setminus V(\cO_{M})^{G,\circ}$
have measure zero: we denote it by $\mu_{V,E}^{M}$. 
\begin{defn}
For a $G$-stable constructible subset $C\subset V_{k}$, we define
\[
\sharp_{\st}^{G}(V,E)_{C}:=q^{d}\cdot\sum_{M\in\GEt}\frac{\mu_{V,E}^{M}(V(\cO_{M})_{C}^{G})}{\sharp C_{G}(H)},
\]
with $d$ the dimension of $V$. 
\end{defn}
With this definition, we obviously have:
\begin{thm}
\label{thm:main general}Theorem \ref{thm:main affine} and Corollary
\ref{cor:main non-log} hold without the assumption that $V$ is affine. \end{thm}
\begin{rem}
Our construction of the measure $\mu_{V,E}^{M}$ and the definition
of $\sharp_{\st}^{G}(V,E)_{C}$ are not completely satisfactory, because
they depend on Theorem \ref{thm:main affine} for the affine case
and then Theorem \ref{thm:main general} is somehow tautology. Therefore
it is an interesting problem to construct the measure intrinsically,
in particular, without gluing affine pieces.
\end{rem}

\section{Linear actions and mass formulas\label{sec:mass formulas}}

In this section, we consider the case $V=\AA_{\cO_{K},\bx}^{n}$ and
compute $M$-stringy point counts $\sharp_{\st}^{M}V$ for $M\in\GEt$.
Then we briefly illustrate how our main results in this case prove
mass formulas by Serre, Bhargava and Kedlaya. However we should note
that it is impossible to locally linearize a given group action, unlike
the case of an algebraically closed base field of characteristic zero.

\subsection{\label{sub:linear}}

We now suppose that $V=\AA_{\cO_{K},\bx}^{n}=\Spec\cO_{K}[x_{1},\dots,x_{n}]$
and that $G$ acts $\cO_{K}$-linearly on it and consider the trivial
log structure $V=(V,E=0)$. For each $M\in\GEt$, let $\Xi_{M}$ be
the tuning module (see section \ref{sub:untwist1}), which is a submodule
of $\Hom_{\cO_{K}}(\cO_{K}[\bx]_{1},\cO_{M})$. We denote the origin
of $V_{k}$ by $o$.
\begin{defn}
We define 
\[
\bv_{V}(M):=\frac{1}{\sharp G}\cdot\length\frac{\Hom_{\cO_{K}}(\cO_{K}[\bx]_{1},\cO_{M})}{\cO_{M}\cdot\Xi_{M}}.
\]
Let $\eta_{k}:V^{\left\langle M\right\rangle }\otimes_{\cO_{K}}k\to V_{k}$
be the base change of $\eta:V^{\left\langle M\right\rangle }\to V$
from $\cO_{K}$ to $k$. We define
\[
\bw_{V}(M):=\dim\eta_{k}^{-1}(o)-\bv_{V}(M).
\]
\end{defn}
\begin{rem}
Our definition of $\bv_{V}$ as a function associated to the $G$-variety
$V$ is identical to the one given in \cite{Wood-Yasuda-I} (see also
Remark \ref{rem: left or right}). Our definition of $\bw_{V}$ is
slightly different from the one given in \cite{Wood-Yasuda-I}. However,
the following lemma shows that they coincide in three important cases.\end{rem}
\begin{lem}
\label{lem:fiber dim}Suppose that one of the following conditions
holds:
\begin{enumerate}
\item $p\nmid\sharp G$,
\item $K=k((t))$ and the $G$-action on $V$ is the base change of one
on $V_{k}$,
\item the $G$-action on $V$ is permutation of coordinates.
\end{enumerate}
For $M\in\GEt$, let $H_{0}\subset G$ be the stabilizer of a geometric
connected component of $\Spec M$, that is, a component of $\Spec M\otimes_{\cO_{K}}\cO_{L}$
with $L$ the maximal unramified extension of $K$. Then 
\[
\dim\eta_{k}^{-1}(o)=\codim((V_{k})^{H_{0}}\subset V_{k}).
\]
\end{lem}
\begin{proof}
Our construction of $V^{|M|}$ and $C^{|M|}$ is compatible with the
base change by $\cO_{K'}/\cO_{K}$ for a finite unramified extension
$K'/K$. Hence we may suppose that a connected component $\Spec L$
of $\Spec M$ is a geometric connected component as well. It follows
that $H_{0}=H$ and that $L$ has the residue field $k$. Let $\eta_{k}:V_{k}^{\left\langle M\right\rangle }\to V_{k}$
be the base change of the morphism $\eta:V^{\left\langle M\right\rangle }\to V$
by $\Spec k\to\Spec\cO_{K}$. The reduced scheme $(V_{k}^{\left\langle M\right\rangle })_{\red}$
associated to $V_{k}^{\left\langle M\right\rangle }$ is the disjoint
union of $[G:H]$ copies of $\AA_{k}^{n}$ and the restriction of
$\eta_{k}$ to each connected component $(V_{k}^{\left\langle M\right\rangle })_{\red}$
is a $k$-linear map. Moreover we can identify $V_{k}^{|M|}$ with
the connected component of $(V_{k}^{\left\langle M\right\rangle })_{\red}$
corresponding to $\Spec L$, which we denote by $V_{k,0}^{\left\langle M\right\rangle }$.
We denote the corresponding component of $V^{\left\langle M\right\rangle }$
by $V_{0}^{\left\langle M\right\rangle }$. To show the lemma, it
suffices to show that the map 
\[
\eta_{k}|_{V_{k,0}^{\left\langle M\right\rangle }}:V_{k,0}^{\left\langle M\right\rangle }\to V_{k}
\]
has image $(V_{k})^{H}$. 

Let $v$ be an arbitrary $k$-point of $V_{k,0}^{\left\langle M\right\rangle }=V_{k}^{|M|}$.
This point lifts to an $\cO_{K}$-point $\tilde{v}$ of $V^{|M|}$
and hence to an $H$-equivariant $\cO_{L}$-point $\hat{v}$ of $V_{0}^{\left\langle M\right\rangle }$.
The image of $\hat{v}$ on $V$ is also $H$-equivariant, and hence
the $k$-point 
\[
\Spec k\hookrightarrow\Spec\cO_{K}\xrightarrow{\hat{v}}V
\]
lies in $(V_{k})^{H}$. From the construction, this $k$-point is
the image of $v$, which shows that the image of $\eta_{k}|_{V_{k,0}^{\left\langle M\right\rangle }}$
is contained in $(V_{k})^{H}$. 

Next let $w$ be an arbitrary $H$-fixed $k$-point of $V_{k}$. From
either of the three conditions in the proposition, there exists an
$H$-fixed $\cO_{K}$-point $\tilde{w}$ of $V$ which is a lift of
$w$. Then the composition 
\[
\hat{w}:\Spec\cO_{L}\to\Spec\cO_{K}\xrightarrow{\tilde{w}}V
\]
is $H$-equivariant. It then lifts to an $H$-equivariant $\cO_{L}$-point
$\check{w}$ of $V_{0}^{\left\langle M\right\rangle }$. The induced
$k$-point 
\[
\Spec k\hookrightarrow\Spec\cO_{L}\xrightarrow{\check{w}}V_{0}^{\left\langle M\right\rangle }
\]
maps to $w$ by $V^{\left\langle M\right\rangle }\to V$. This proves
that the image of $\eta_{k}|_{V_{k,0}^{\left\langle M\right\rangle }}$
contains $(V_{k})^{H}$ and completes the proof of the lemma. 
\end{proof}
Let $(V^{|M|},E^{|M|})$ be the log structure on $V^{|M|}$ defined
as in section \ref{sub:embed}, where we do not need the normalization
as $V^{|M|}=\AA_{\cO_{K},\by}^{n}$ is clearly normal. 
\begin{lem}
Regarding $V_{k}^{|M|}=\AA_{k,\by}^{n}$ as a prime divisor on $V^{|M|}$,
we have 
\[
E^{|M|}=-\bv_{V}(M)\cdot V_{k}^{|M|}.
\]
\end{lem}
\begin{proof}
This is a special case of \cite[Lemma 6.5]{Yasuda:2014fk2} (except
that $k$ is finite in the present paper, while it is algebraically
closed in the cited paper). The outline of the proof is as follows.
Let $m$ be the quotient of $\cO_{M}$ by the Jacobson radical. Namely
$\Spec m$ is the union of the closed points of $\Spec\cO_{M}$ with
reduced structure. We put $V_{m}^{\left\langle M\right\rangle }:=V^{\left\langle M\right\rangle }\otimes_{\cO_{M}}m$.
Let $\Spec L$ be a connected component of $\Spec M$ and $H\subset G$
its stabilizer. Let $\delta_{L/K}$ be the different exponent of $L/K$
(the different of $L/K$ is $\fm_{L}^{d_{L/K}}$). If $(V^{\left\langle M\right\rangle },E^{\left\langle M\right\rangle })$
is the induced log structure on $V^{\left\langle M\right\rangle }$,
then 
\[
E^{\left\langle M\right\rangle }=-(\sharp H\cdot\bv_{V}(M)+\delta_{L/K})V_{m}^{\left\langle M\right\rangle },
\]
where the term $\sharp H\cdot\bv_{V}(M)$ is a contribution of the
morphism $V^{\left\langle M\right\rangle }\to V\otimes\cO_{M}$ and
the term $\delta_{L/K}$ is a contribution of the morphism $V\otimes\cO_{M}\to V$.
Now we have
\[
E^{|M|}=\frac{1}{\sharp G}\left(\theta_{*}E^{\left\langle M\right\rangle }\right)=-\bv_{V}(M)\cdot V_{k}^{|M|}
\]
with $\theta$ the natural morphism $V^{\left\langle M\right\rangle }\to V^{|M|}$.\end{proof}
\begin{prop}
\label{prop:linear McKay}We have 
\begin{gather*}
\sharp_{\st}^{M}V=\frac{q^{n-\bv_{V}(M)}}{\sharp C_{G}(H)},\,\sharp_{\st}^{G}V=\sum_{M\in\GEt}\frac{q^{n-\bv_{V}(M)}}{\sharp C_{G}(H)},\\
\sharp_{\st}^{M}(V)_{o}=\frac{q^{\bw_{V}(M)}}{\sharp C_{G}(H)}\text{ and }\sharp_{\st}^{G}(V)_{o}=\sum_{M\in\GEt}\frac{q^{\bw_{V}(M)}}{\sharp C_{G}(H)}.
\end{gather*}
\end{prop}
\begin{proof}
If we write $\bv_{V}(M)=s/r$ with integers $s\ge0$ and $r>0$, then
$\cO_{V^{|M|}}(r(K_{V^{|M|}}+E^{|M|}))$ has a generator $\varpi^{s}(dy_{1}\wedge\cdots\wedge dy_{n})^{\otimes r}$
with $\varpi$ a uniformizer of $K$. We have 
\begin{align*}
\sharp_{\st}^{M}V & =\frac{\sharp_{\st}(V^{|M|},E^{|M|})}{\sharp C_{G}(H)}\\
 & =\frac{q^{n}}{\sharp C_{G}(H)}\int_{\cO_{K}^{n}}|\varpi^{s}|^{1/r}\\
 & =\frac{q^{n-s/r}}{\sharp C_{G}(H)}\cdot\mu_{K^{n}}(\cO_{K}^{n})\\
 & =\frac{q^{n-s/r}}{\sharp C_{G}(H)}.
\end{align*}
showing the first two equalities of the proposition. If we put $a:=\dim\eta_{k}^{-1}(o)$,
the subset $C^{|M|}\subset V_{k}^{|M|}$ is a linear subspace of dimension
$a$. Therefore
\begin{align*}
\sharp_{\st}^{M}(V)_{o} & =\frac{\sharp_{\st}(V^{|M|},E^{|M|})_{C^{|M|}}}{\sharp C_{G}(H)}\\
 & =\frac{q^{n}}{\sharp C_{G}(H)}\int_{\cO_{K}^{a}\times\fm_{K}^{n-a}}|\varpi^{s}|^{1/r}\\
 & =\frac{q^{n-s/r}}{\sharp C_{G}(H)}\cdot\mu_{K^{n}}(\cO_{K}^{a}\times\fm_{K}^{n-a})\\
 & =\frac{q^{a-s/r}}{\sharp C_{G}(H)}.
\end{align*}
This shows the last two equalities of the proposition.
\end{proof}

\subsection{\label{sub:mass}}

Serre \cite{MR500361} proved a beautiful mass formula: for each integer
$n\ge2$, 
\begin{eqnarray*}
\sum_{\substack{L/K\text{: tot. ram.}\\
{}[L:K]=n
}
}\frac{q^{-d_{L/K}}}{\sharp\Aut(L/K)} & = & q^{1-n},
\end{eqnarray*}
where $L$ runs over the isomorphism classes of totally ramified field
extensions of a fixed local field $K$ with $[L:K]=n$, $d_{L/K}$
is the discriminant exponent of $L/K$ (the discriminant of $L/K$
is $\fm_{K}^{d_{L/K}}$) and $\Aut(L/K)$ the group of $K$-automorphisms
of $L$. Bhargava \cite{MR2354798} proved a similar formula: for
each $n\ge2$, denoting by $\nEt$ the set of isomorphism classes
of étale $K$-algebras of degree $n$, we have 
\[
\sum_{L\in\nEt}\frac{q^{-d_{L/K}}}{\sharp\Aut(L/K)}=\sum_{i=0}^{n-1}P(n,n-i)q^{-i},
\]
where $P(n,n-i)$ is the number of partitions of the integer $n$
into exactly $n-i$ parts.

\subsection{\label{sub:corresp etale}}

Let $S_{n}$ be the $n$-th symmetric group acting on $\{1,\dots,n\}$.
We embed $S_{n-1}$ into $S_{n}$ as the stabilizer subgroup of $1$.
Let $\nEt$ be the set of isomorphism classes of étale $K$-algebras
of degree $n$. We have a bijection 
\[
\SnEt\to\nEt,\,M\mapsto M^{S_{n-1}}.
\]
Moreover the automorphism group of $M$ as an $S_{n}$-étale $K$-algebra
is isomorphic to the automorphism group of the étale $K$-algebra
$M^{S_{n-1}}$. 

Suppose that $S_{n}$ acts on $V=\AA_{\cO_{K}}^{2n}$ by the direct
sum of two copies of the standard permutation representation. Wood
and Yasuda \cite{Wood-Yasuda-I} showed 
\[
d_{M^{S_{n-1}}/K}=\bv_{V}(M).
\]
Therefore the left hand side of Bhargava's formula is written as 
\[
\sum_{M\in\SnEt}\frac{q^{-\bv_{V}(M)}}{\sharp C_{S_{n}}(H)},
\]
where $H\subset S_{n}$ is the stabilizer of a component of $\Spec M$.

The quotient variety $V/S_{n}$ is identical to the $n$-th symmetric
product of $\AA_{\cO_{K}}^{2}$ over $\cO_{K}$. Let $\Hilb^{n}(\AA_{\cO_{K}}^{2})$
be the Hilbert scheme of $n$ points of $\AA_{\cO_{K}}^{2}$ defined
relatively over $\cO_{K}$, which is a smooth $\cO_{K}$-variety.
From \cite[7.4.6]{MR2107324} the Hilbert-Chow morphism $\Hilb^{n}(\AA_{\cO_{K}}^{2})\to V/S_{n}$
is proper, birational and crepant. Therefore 
\[
\sharp_{\st}(V/S_{n})=\sharp_{\st}\Hilb^{n}(\AA_{\cO_{K}}^{2})=\sharp\Hilb^{n}(\AA_{\cO_{K}}^{2})(k).
\]
Using a stratification of $\Hilb^{n}(\AA_{k}^{2})$ into affine spaces
(or Gröbner basis theory), we can count the $k$-points of $\Hilb^{n}(\AA_{k}^{2})$
and get 
\[
\sharp\Hilb^{n}(\AA_{\cO_{K}}^{2})(k)=\sum_{i=0}^{n-1}P(n,n-i)q^{2n-i}.
\]
Proposition \ref{prop:linear McKay} gives
\[
\sum_{M\in\SnEt}\frac{q^{2n-\bv_{V}(M)}}{\sharp C_{S_{n}}(H)}=\sharp_{\st}^{S_{n}}V=\sharp_{\st}(V/S_{n})=\sharp_{\st}\Hilb^{n}(\AA_{\cO_{K}}^{2})=\sum_{i=0}^{n-1}P(n,n-i)q^{2n-i}.
\]
Dividing these by $q^{2n}$, we reprove Bhargava's mass formula as
a consequence of the wild McKay correspondence. This computation will
be revisited in \cite{Wood-Yasuda-II} in relation to dualities discussed
there. Kedlaya \cite{MR2354797} obtained a similar formula for the
group of signed permutation matrices. In \cite{Wood-Yasuda-II}, Kedlaya's
formula is also deduced from the wild McKay correspondence in a very
similar way except the case of residual characteristic two.

We can prove Serre's mass formula from Bhargava's mass formula, for
instance, by using the exponential formula in a direction opposite
to the one in \cite{MR2354797}. Let us write
\[
N(K,n):=\sum_{\substack{L/K\text{: tot. ram.}\\
{}[L:K]=n
}
}\frac{q^{-d_{L/K}}}{\sharp\Aut(L/K)}\text{ and }M(K,n):=\sum_{L\in\nEt}\frac{q^{-d_{L/K}}}{\sharp\Aut(L/K)}.
\]
Kedlaya \cite[pages 7--8]{MR2354797} showed 
\[
\sum_{n=0}^{\infty}M(K,n)x^{n}=\exp\left(\sum_{n=1}^{\infty}\frac{x^{n}}{n}\sum_{f\mid n}\frac{N(K_{f},n/f)}{f}\right),
\]
where $K_{f}$ is the unramified extension of $K$ of degree $f$.
He used Serre's formula for $N(K_{f},n/f)$ to get Bhargava's formula.
It is easy to show that given the values of $M(K_{f},n/f)$ for all
$f$ and $n$, then the equality above determines the values of $N(K_{f},n)$
for all $f$ and $n$. In particular, $M(K,n)$ must be equal to $q^{1-n}$
and Serre's formula holds. 

\bibliographystyle{alpha}
\bibliography{/Users/Takehiko/Dropbox/Math_Articles/mybib}

\end{document}